\documentclass[11pt,a4paper,twoside,reqno]{amsart}

\usepackage{amssymb}
\usepackage{latexsym}

\usepackage{graphicx}
\usepackage{a4wide}

\usepackage{amsmath}
\usepackage{amsthm}

\usepackage{color}
\usepackage{caption}
\usepackage{subcaption}
\usepackage{hyperref}
\usepackage{multirow}
\usepackage{booktabs}
\usepackage{mathdots}

\newtheorem{theorem}{Theorem}
\newtheorem{lemma}{Lemma}
\newtheorem{remark}{Remark}
\newtheorem{definition}{Definition}

\newcommand{\norm}[1]{\left\| #1 \right\|}

\renewcommand{\i}{\ifmmode\mathit{\mathchar"7010 }\else\char"10 \fi}
\renewcommand{\j}{\ifmmode\mathit{\mathchar"7011 }\else\char"11 \fi}
\newcommand{\R}{\mathbb{R}}

\begin{document}\large

\title{Computation of Eigenvalues for Nonlocal Models by Spectral Methods}%

\author[L. Lopez]{Luciano Lopez}

\author[S. F. Pellegrino]{Sabrina Francesca Pellegrino
}

\address[L. Lopez]{Dipartimento di Matematica, Universit\`a degli Studi di Bari Aldo Moro, via E. Orabona 4, 70125 Bari, Italy}
\email{luciano.lopez@uniba.it}

\address[S. F. Pellegrino]{Dipartimento di Management, Finanza e Tecnologia, Universit\`a LUM Giuseppe Degennaro, S.S. 100 Km 18 - 70010 Casamassima (BA), Italy}
\email{pellegrino@lum.it}

\begin{abstract}
The purpose of this work is to study spectral methods to approximate the eigenvalues of nonlocal integral operators. Indeed, even if the spatial domain is an interval, it is very challenging to obtain closed analytical expressions for the eigenpairs of peridynamic operators. Our approach is based on the weak formulation of eigenvalue problem and in order to compute the eigenvalues we consider an orthogonal basis consisting of a set of Fourier trigonometric or Chebyshev polynomials. We show the order of convergence for eigenvalues and eigenfunctions in $L^2$-norm, and finally, we perform some numerical simulations to compare the two proposed methods.
\end{abstract}

\maketitle





\section{Introduction}
\label{intro}
Eigenvalues and eigenfunctions are used widely in science and engineering. They have many applications, particularly in physics, where they are usually related to vibrations. Indeed, eigenfunctions can provide the direction of spread of data, while eigenvalues represent the intensity of spread in the direction of the corresponding eigenfunction.

Eigenvalue problems occur naturally in the vibration analysis of mechanical structures with many degrees of freedom and in the context of non linear and nonlocal wave equations. The eigenvalues are used to determine the natural frequencies of vibration, and the relative eigenfunctions determine the shape of these vibrational modes. The orthogonality properties of the eigenfunctions allow a differential equation to be decoupled into a certain number of independent directions whose linear combination help to better understand the system under consideration. 

In particular, eigenvalue problems can be achieved when someone looks for the solutions of non linear problems by using the method of separation of variables.

The eigenvalue problem of complex structures is often solved using finite element analysis or spectral techniques.

Moreover, they have applications in other areas of applied mathematics such as elastic theory, deformation, finance, quantum mechanics, image recognition and dimensional reduction.

In this work, we restrict our attention to the applications of the eigenvalue problems to nonlocal peridynamic framework.

Peridynamics is a nonlocal theory for elasticity and dynamic fracture analysis consisting in a second order in time partial integro-differential equation introduced by Silling in~\cite{Silling_2000}. It is suitable to solve problems with spontaneous formation of discontinuities or singularities as it avoids the use of spatial partial derivatives. It is a nonlocal theory as the interactions between material points have a long-range: this means that a material particle interacts with other units in its neighborhood directly across finite distance called horizon (see~\cite{SILLING_2005_2,MadenciOterkus2017,Delia2017,Bobaru_2009,Madenci2015,MadenciOterkus,KILIC2010,jafari}). 

The horizon represents the locality of the model, and one can prove that when the horizon goes to zero, the peridynamic equation approaches the wave equation (see~\cite{L,alebrahim,Alebrahim2021,chen}). Additionally, the peridynamic operator has been used to approximate the fractional Laplacian operator (see~\cite{DELIA2013}).

In peridynamic setting, the eigenvalue analysis predicts the theoretical buckling strength of a linear elastic structure and the eigenvalue problem is correlated to the static structural analysis of the system under consideration (see~\cite{heo2020,jafari,heo}). It can be useful to explore the deformation mechanism and the engineering properties of an elastic material. Moreover, it contributes to the determination of the preload in the nonlinear analysis can be utilized to investigate the effect of crack length, crack orientation and plate thickness. Indeed, the eigenfunctions represent the buckling mode that the material under consideration performs.

Several numerical methods are proposed to study the Laplacian eigenvalue problem. In particular the weak Galerkin finite element method and its accelerated version such as the two-grid and two-space methods are highly flexible and efficient methods used for the computation of the Laplacian eigenvalues, (see~\cite{xie,zhai2019}).

Our work is inspired by some studies of weak eigenvalues formulation for some local and nonlocal peridynamic operators (see~\cite{alali,aksoylu,du}).

Applying the nonlinear mechanical properties in practice is a difficult issue, and, as a result, finite element method is frequently adopted for the spatial discretization. Additionally, for the simple geometric nonlinearities, analytical solution can be obtained.

For radially symmetric kernels, the evaluation of eigenvalues of linear nonlocal peridynamic operators under periodic boundary conditions can be used to obtain a spectral approximation of such operators. Following this idea, an hybrid spectral discretization obtained by combing truncated series expansion and high order Runge-Kutta ODEs solvers is proposed in~\cite{du}.

In~\cite{aksoylu}, the authors define local periodic and antiperiodic operators to enforce boundary conditions to the linear one dimensional peridynamic model, derive an explicit expression of their eigenvalues in terms of the horizon and study the conditioning and error analysis of such operators. For the numerical tests they use the Nystr\"om method with the trapezoidal rule for the discretization.

An important issue in peridynamics is finding the solution of time dependent initial and boundary value problems. In~\cite{Galvanetto2016,Galvanetto2018}, the authors proposed coupled meshless finite point and finite element methods to discretize the peridynamic equation. Recently, spectral methods based on the discrete Fourier transform and convolution properties have been proposed to efficiently approximate the solution, (see~\cite{CFLMP,LP,Jafarzadeh,Jafarzadeh2021,LP2021,Bobaru2021}). A different approach can be developed by using eigenvalues, (see~\cite{heo2020}). Indeed, the computation of the eigenvalues of nonlocal operators can be used to solve periodic nonlocal time dependent problems. In particular, in~\cite{alali}, the authors proposed efficient and accurate spectral solvers based on the hypergeometric representation of the Fourier multipliers to evaluate the eigenvalues of the linear peridynamic integral operator.

In this paper, we focus on the weak formulation of a class of nonlocal peridynamic models and we study the spectrum properties of the involved operator. We show that the solution of the problem can be obtained as linear combination of a certain number of vibrational modes. We consider two different orthogonal basis to compute the eigenvalues and eigenfunctions of the peridynamic operator. The first one consists of Fourier trigonometric polynomials and is suitable for problems with periodic boundary conditions, while the second one does not require the periodicity condition and is given by the Chebyshev polynomials.


Let $\Omega$ be a one-dimensional bounded domain, then, the general form of the nonlocal peridynamic operator is the following
\begin{equation}
\label{eq:Lf}
\mathcal{L}w(x) = \int_{\Omega\cap B_{\delta}(x)} f(x'-x, w(x')-w(x))dx', \qquad x\in\Omega,
\end{equation}
where $B_{\delta}(x) = \{x'\in\Omega\, :\, |x-x'|\le\delta \}$, $\delta>0$ is the horizon, $f$ is the pairwise force function, which describes the interaction between two material points, and the unknown $w$ represents the displacement field.

In what follows we restrict our attention to the case in which the evolution of the material body is given by a linear pairwise force function of convolution type in separable form
\begin{equation}
\label{eq:f}
f(x'-x, w(x') - w(x)) = C(x'-x) \left(w(x')-w(x)\right),
\end{equation} 
where the function 
$C\in L^\infty(\Omega)\cap C^1(\overline{\Omega})$
 is a bounded positive even function, i.e. $C(x'-x)=C(x-x')$, called micromodulus function, such that $C(x'-x)\equiv 0$, for all $x$, $x'\in\Omega$ such that $|x-x'|>\delta$. Examples of micromodulus functions for linear microelastic materials can be found in~\cite{WECKNER2005705}.

Thus the peridynamic operator becomes
\begin{equation}
\label{eq:LC}
\mathcal{L}w(x) = \int_{\Omega\cap B_{\delta}(x)} C(x'-x) \left(w(x')-w(x)\right)dx'.
\end{equation}


In the present paper we aim at studying the following eigenvalue problem
\begin{equation}
\label{eq:eigpb}
-\mathcal{L}w(x) = \lambda w(x),
\end{equation}
with 
periodic boundary conditions on a one-dimensional bounded domain $\Omega$, by means of spectral Fourier and Chebyshev approximation, where $\mathcal{L}$ is defined in~\eqref{eq:LC}. 



The paper is organized as follows. Section~\ref{sec:theoretic} recalls the main theoretical results about the existence of eigenvalues for the variational formulation of the problem~\eqref{eq:eigpb}. Section~\ref{sec:Spectral} is devoted to the description of a numerical method for solving the eigenvalue problem by using the Fourier trigonometric and the Chebyshev polynomials. Moreover, this section contains some estimates about the distance between the continuous eigenpairs and their discrete approximation. Section~\ref{sec:numerics} collects the numerical simulations and finally, Section~\ref{sec:concl} concludes the paper.

\section{Main theoretical results}
\label{sec:theoretic}

We recall the main theoretical results about the functional space in which we will work and the characterization of eigenvalues in terms of the min-max principle. To do so, we start by introducing the nonlocal integration by part rule.
\begin{lemma}
The following equality holds
\begin{equation}
\label{eq:pp}
\int_{\Omega} \mathcal{L}w(x) v(x)dx = -\frac{1}{2}\int_\Omega\int_\Omega C(x'-x)\left(w(x')-w(x)\right)\left(v(x')-v(x)\right)dx'dx. 
\end{equation}
\end{lemma}

\begin{proof}
Using the symmetry of $C$ and by relabeling the variable, we find
\[
\int_\Omega \mathcal{L}w(x) v(x)dx =\int_\Omega\int_\Omega C(x'-x)\left(w(x)-w(x')\right)v(x') dxdx'
\]
Moreover, by applying Fubini-Tonelli's theorem there holds
\[
\int_\Omega \mathcal{L}w(x) v(x)dx =-\int_\Omega\int_\Omega C(x'-x)\left(w(x')-w(x)\right)v(x') dx'dx.
\]
Therefore,
\begin{align*}
\int_\Omega \mathcal{L}w(x) v(x)dx &= \frac{1}{2} \int_\Omega \mathcal{L}w(x) v(x)dx + \frac{1}{2} \int_\Omega \mathcal{L}w(x) v(x)dx\\
&=-\frac{1}{2}\int_\Omega\int_\Omega C(x'-x)\left(w(x')-w(x)\right)\left(v(x')-v(x)\right)dx'dx. 
\end{align*}
\end{proof}

In what follows, we use $(\cdot , \cdot)$ and $\norm{\cdot}$ to denote the inner product and the norm of the Hilbert space $L^2(\Omega)$, respectively, namely
\[\left(w,v\right) = \int_\Omega w(x)v(x) dx,\quad \norm{w}^2 = \left(w,w\right),\qquad w,\,v\in L^2(\Omega).\]

Let $V\subset L^2(\Omega)$ be a suitable Banach space, then, using~\eqref{eq:pp}, we define the {\em energetic extension} of the peridynamic operator $\mathcal{L}: V\to V'$ as follows
\begin{align}
\label{eq:energ-ext}
\left(\mathcal{L}w,v\right)&= \int_\Omega \mathcal{L}w(x) v(x) dx\\
&= -\frac{1}{2}\int_\Omega\int_\Omega C(x'-x)\left(w(x')-w(x)\right)\left(v(x')-v(x)\right)dx'dx,\notag
\end{align}
where $V'$ is the dual space of $V$.

The choice of the function space $V$ depends on the degree of singularity of the micromodulus function $C$ at $x'-x=0$. In case of strong singularity, one need to work with fractional Sobolev spaces, otherwise, we can work in the framework of $L^p$ spaces, for $p\ge 2$ (see~\cite{Emmrich_Puhst_2015}).

In particular, in what follows we consider the functional space
\begin{equation}
\label{eq:V}
V=\{v\in L^2(\Omega) : |v|_V^2<\infty\},
\end{equation}
where $|\cdot|_V$ is the seminorm
\begin{equation}
\label{eq:seminorm}
|v|_V^2 = \int_\Omega\int_\Omega C(x'-x)|v(x')-v(x)|^2dx'dx,\qquad v\in L^2(\Omega).
\end{equation}

We have that $V$ is an Hilbert space endowed with the norm $\norm{\cdot}_V = \norm{\cdot}_{L^2(\Omega)} + |\cdot|_V$. Additionally, due to the assumptions on the micromodulus function, the $V$-norm is equivalent to the $L^2(\Omega)$-norm, in fact, for every $v\in V$ it holds
\[
\norm{v}_{L^2(\Omega)} \le \norm{v}_V\le (1+M) \norm{v}_{L^2(\Omega)},
\]
for some positive constant $M$.

As a consequence, the following theorem holds.
\begin{theorem}
The embedding $V\hookrightarrow L^2(\Omega)$ is continuous.
\end{theorem}

In this context, the eigenvalue problem~\eqref{eq:eigpb} has the following variational formulation: find $\lambda\in\R$ and $w\in V$, $w\ne0$ such that
\begin{equation}
\label{eq:variat}
\left(\mathcal{L}w,v\right) = \lambda (w,v),\qquad v\in V.
\end{equation}

It is obtained by multiplying by a test function the equation~\eqref{eq:eigpb} and by using the nonlocal integration by part rule~\eqref{eq:pp}.

Existence and uniqueness of weak solution $w\in V$ of the problem~\eqref{eq:eigpb} is a consequence of the Lax-Milgram theorem (see~\cite{evans}).

\begin{remark}
In this paper we focus on the weak formulation~\eqref{eq:variat} of~\eqref{eq:eigpb}, as it fits very well with the spectral methods we proposed. As a consequence, in what follows, unless different specified, when we talk about exact eigenvalues, we always refer to the weak ones. A classical eigenvalue is also a weak eigenvalue, but the reverse is not true in general. However, in our setting, using classical results of functional analysis for elliptic operators, one can prove the convergence of the weak problem~\eqref{eq:variat} to the classical one~\eqref{eq:eigpb}.
\end{remark}

A useful characterization of eigenvalues in terms of the {\em Rayleigh quotient} can be derived by using the so-called {\em min-max principle} for elliptic operators, which can be applied to self-adjoint operators in a Hilbert space. (see~\cite{evans}).

Let $\lambda_1\le\lambda_2\le\cdots\le\lambda_k\le\cdots$ be the eigenvalues of $-\mathcal{L}$ defined in~\eqref{eq:LC}. The corresponding eigenfunctions $\{w_k\}_k$ form a complete orthogonal basis in $L^2(\Omega)$.

\begin{definition}
Let $w\in L^2(\Omega)$, with $w\ne 0$. Then, the {\em Rayleigh quotient} associated to the peridynamic operator $\mathcal{L}$ is defined as follows
\begin{equation}
\label{eq:RQ}
R(w) = \frac{\left(\mathcal{L}w, w\right)}{\norm{w}^2}.
\end{equation}
\end{definition}

We report here a well-known result on the Rayleigh quotient.
\begin{theorem}
\label{th:minmax}
The $k$-th eigenvalue $\lambda_k$ of $\mathcal{L}$ satisfies the following condition
\begin{equation}
\label{minmax}
\lambda_k = \min_{W_k} \max_{\stackrel{w\in W_k}{w\ne0}} R(w),
\end{equation}
where $W_k$ is a $k$-dimensional subspace.
\end{theorem}

\section{Spectral approximation}
\label{sec:Spectral}

In this section, we will consider two methods to approximate the eigenvalues of the nonlocal peridynamic operator. They deal with the weak formulation of the corresponding eigenvalue problem and use, as basis for the approximation, the Fourier trigonometric and the Chebyshev polynomials. 

For each method, we will discuss its approximation properties and establish some convergence results for the nonlocal peridynamic problem.

As pointed out in~\cite{alali}, in the case of periodic domains, the classical eigenvalues of the peridynamic operator $\mathcal{L}$ are given by its Fourier multipliers defined through the Fourier transform
\begin{equation}
\label{eq:Fouriermultip}
\lambda(\nu) = \int_{B_{\delta}(0)} C(y)\left( 1-cos(\nu y)\right)dy.
\end{equation}

In order to approximate the eigenvalues in~\eqref{eq:Fouriermultip} one can use a quadrature formula to compute numerically the integral, and the order of accuracy of the method provides the discretization error in the space variable. Instead, in this work, we compute an approximation of the eigenvalues of the peridynamic operator $\mathcal{L}$ by means of the weak formulation~\eqref{eq:variat} of the problem~\eqref{eq:eigpb}. As we will show in the next section, our approach can benefit from the properties of the fast Fourier Transform (FFT), so that the dominant cost of the computation is those of a spectral solver.

More precisely, to obtain a numerical method for the problem~\eqref{eq:variat}, we replace the infinite dimensional space $V$ by a finite dimensional subspace $V_N$ for some integer $N>0$. So, the problem becomes to find $\lambda_N\in\R$ and $w_N\in V_N$, $w_N\ne0$ such that
\begin{equation}
\label{eq:variat-discr}
\left(\mathcal{L}w_N,v_N\right)= \lambda_N (w_N,v_N),\quad\text{for all $v_N\in V_N$}.
\end{equation}

As we will see later, since the problem~\eqref{eq:variat-discr} is finite dimensional, it can be rewritten as a matrix eigenvalue problem. In particular, different approximations correspond to different choices of the subspace $V_N$.

Let $\Omega=[a,b]$ be the computational domain. 
We fix $N>0$ an even integer number 
 and let $\{\psi_k(x)\}_{|k|\le N/2}$ be an orthogonal basis in the space $V\subset L^2([a,b])$, so that $V_N = Span\{\psi_{-N/2}(x),\dots,\psi_{N/2}(x)\}$.

Then, we recall that the variational formulation of the eigenvalue problem~\eqref{eq:eigpb} is equivalent to finding a function $w$, such that 
\begin{equation}
\label{eq:variational}
\int_a^{b}\left(-\int_{\Omega\cap B_{\delta}(x)} C(x-x')\left(w(x')-w(x)\right)dx'\right)\psi_k(x)\,dx =\int_a^b \lambda w(x)\psi_k(x)\,dx.
\end{equation}


We seek $w_N\in V_N$ an approximation of $w(x)$ in form of linear combination of the basis $\{\psi_k(x)\}_{|k|\le N/2}$, that is
\begin{equation}
\label{eq:Ftransf}
w_N(x)=\sum_{|h|\le N/2}\tilde{w}_h\psi_h(x),
\end{equation}
where $\tilde{w}_h$ are the coefficients of the linear combination. 

Substituting~\eqref{eq:Ftransf} into~\eqref{eq:variational}, we get

\begin{align}
\label{eq:subs}
&\int_a^{b}\left(-\int_{\Omega\cap B_{\delta}(x)}C(x-x')\left(\sum_{|h|\le N/2}\tilde{w}_h(\psi_h(x')-\psi_h(x))\right)dx'\right)\psi_k(x)\,dx\\
&\quad=\int_a^{b}\lambda_N\sum_{|h|\le N/2}\tilde{w}_h\psi_h(x)\psi_k(x)\,dx\notag
\end{align}
 for all $k=-N/2,\cdots, N/2$.

Notice that we can rewrite the left hand side of~\eqref{eq:subs} as follows
\begin{align*}
&\int_a^{b}\left(-\int_{\Omega\cap B_{\delta}(x)}C(x-x')\left(\sum_{|h|\le N/2}\tilde{w}_h(\psi_h(x')-\psi_h(x))\right)dx'\right)\psi_k(x)\,dx\\
&\quad =\sum_{|h|\le N/2}\tilde{w}_h\left(-\int_a^{b}\psi_k(x)\int_{\Omega\cap B_{\delta}(x)}C(x-x')(\psi_h(x')-\psi_h(x))dx'dx\right),\notag
\end{align*}
while the right hand side of~\eqref{eq:subs} becomes
\begin{equation*}
\int_a^{b}\lambda_N\sum_{|h|\le N/2}\tilde{w}_h\psi_h(x)\psi_k(x)\,dx =\lambda_N\sum_{|h|\le N/2}\tilde{w}_h\left(\psi_h(x),\psi_k(x)\right).\notag
\end{equation*}

Hence, in matrix notation, equation~\eqref{eq:subs} gives the following eigenvalue problem
\begin{equation}
\label{eq:geneigpb}
A\tilde{w} = \lambda_N M \tilde{w},
\end{equation}
where $A=(a_{hk})_{h,k}$ and $M=(m_{hk})_{h,k}$, with
\begin{align*}
a_{hk}&=-\int_a^{b}\psi_k(x)\int_{\Omega\cap B_{\delta}(x)}C(x-x')(\psi_h(x')-\psi_h(x))\,dx'dx\\
&=-\int_a^{b}\psi_k(x)\int_{\Omega\cap B_{\delta}(x)}C(x-x')\psi_h(x')dx'dx+\beta\int_a^b\psi_k(x)\psi_h(x)dx,
\end{align*}
\[m_{hk}=\int_a^{b}\psi_h(x)\psi_k(x)\,dx,\]
and $\beta=\int_{B_{\delta}(0)}C(x)dx$.

Therefore, the approximate problem~\eqref{eq:variat-discr} is equivalent to the matrix eigenvalue problem~\eqref{eq:geneigpb}.

We can observe that, due to the orthogonality of the basis $\{\psi_k(x)\}_{|k|\le N/2}$, the matrix $M$ is diagonal with non-zero diagonal entries.

Additionally, we can order the approximate eigenvalues of~\eqref{eq:variat-discr} as follows
\[\lambda_N^1\le\lambda_N^2\le\cdots\le\lambda_N^N.\]
They fulfill an analogue characterization of the continuous eigenvalues:
\[
\lambda_N^k = \min_{W_{k,N}\subset V_N}\quad\max_{\stackrel{w_N\in W_{k,N}}{w_N\ne0}}\frac{\norm{w_N}_{W_{k,N}}^2}{\norm{w_N}_{L^2(\Omega)}^2},
\]
where $W_{k,N}$ is a $k$-dimensional subspace of $V_N$.

Since $V_N\subset V$, as a consequence of Theorem~\ref{th:minmax}, the following lemma holds.
\begin{lemma}
\label{lm:sort}
Let $V_N\subset V$. If $\lambda_k$ and $\lambda_N^k$ are the $k$-th smallest eigenvalues of~\eqref{eq:variat} and~\eqref{eq:variat-discr}, respectively, then $\lambda_k\le\lambda_N^k$.
\end{lemma}

This lemma will be useful to prove the error estimate of the distance between the continuous and the discrete eigenvalues.

\subsection{Convergence of the spectral method}

This section is devoted to the proof of the following convergence result.
\begin{theorem}
\label{th:conv}
If $(\lambda_N^k, w_N^k)$ is the $k$-th smallest eigenpair corresponding to the discrete problem~\eqref{eq:variat-discr}. Then, there exists a positive constant $M = M(\Omega)$ independent of $N$ such that
\begin{equation}
\label{eq:dist-eig}
\lambda_k\le \lambda_N^k\le \lambda_k\left(1+\frac{M}{N^2}\right),
\end{equation}
where $\lambda_k$ is the $k$-th smallest eigenvalue of the continuous problem~\eqref{eq:variat}.
\end{theorem}

In order to prove the theorem, we need to introduce some definitions and to recall some regularity results.

Let start by introducing the projection operator with respect to the $V$-norm and the interpolant operator.

\begin{definition}
The projection operator is the function $P_N: V\to V_N$ which associates to every $v\in V$ the only function $P_N v\in V_N$ satisfying the following orthogonality condition
\begin{equation}
\label{eq:ort1}
\left(\mathcal{L}(v-P_N v),v_N\right)=0,\quad\text{for all $v_N\in V_N$},
\end{equation}
or equivalently
\begin{equation}
\label{eq:ort2}
\norm{v-P_N v}_V = \inf_{v_N\in V_N} \norm{v-v_N}_V
\end{equation}
\end{definition}

\begin{definition}
Let $\{x_j\}_{j=1}^N$ be a partition of $\Omega$. The interpolant operator is the function $I_N: V\to V_N$ which associates to every $v\in V$ the interpolant of degree $N$ defined as follows
\begin{equation}
\label{eq:IN}
I_N v = \sum_{j=1}^N v(x_j)\psi_j,
\end{equation}
where $\{\psi_j\}_{j=1}^N$ is the orthogonal basis in $V$ we fixed for the approximation.
\end{definition}

\begin{remark}
Condition~\eqref{eq:ort2} ensures that
\begin{equation}
\label{eq:est1}
\norm{v-P_N v}_V \le \norm{v-I_N v}_V,\quad v\in V.
\end{equation}
\end{remark}

We now prove an interpolation estimate.
\begin{lemma}
Let $I_N$ be the interpolant operator defined in~\eqref{eq:IN}. Then, there exists a positive constant $M=M(\Omega)$ independent of $N$ such that
\begin{equation}
\label{eq:est2}
\norm{v- I_N v}_V \le \frac{M}{N^2}\norm{v}_{L^2(\Omega)}.
\end{equation}
\end{lemma}

\begin{proof}
We have
\begin{align*}
\norm{v- I_N v}_V^2 &\le \int_\Omega\int_\Omega C(x'-x)\left((v-I_N v)(x')-(v-I_N v)(x)\right)^2dx'dx\\
&\le M\norm{v-I_N v}_{L^2(\Omega)}^2\le\frac{M}{N^4}\norm{v}_{L^2(\Omega)}^2,
\end{align*}
where the last inequality is a consequence of the spectral properties of the interpolant operator (see\cite{Canuto2006}).
\end{proof}

We recall the regularity result, (see~\cite{Caffarelli}): given a function $f\in L^2(\Omega)$, let $w\in V$ be the solution of the boundary-valued problem
\[
\begin{cases}
-\mathcal{L}w = f &\text{in $\Omega$}\\
w=0 &\text{on $\partial \Omega$}.
\end{cases}
\]
Then, the following regularity estimate holds
\begin{equation}
\label{eq:reg}
\norm{w}_{L^2(\Omega)} \le M \norm{f}_{L^2(\Omega)},
\end{equation}
for some positive constant $M=M(\Omega)$.

The next lemma studies the $L^2$-convergence of the projection operator.
\begin{lemma}
\label{lm:L2proj}
Let $v\in V$ be an eigenfunction of the problem~\eqref{eq:variat}. Then, there is a positive constant $M$ independent of $N$ such that
\begin{equation}
\label{eq:est3}
\norm{v-P_N v}_{L^2(\Omega)} \le \frac{M}{N^2}.
\end{equation}
\end{lemma}

\begin{proof}
Let $w\in V$ be the weak solution of the problem
\begin{equation*}
-\mathcal{L}w = v- P_N v.
\end{equation*}
Thus, for every $\tilde{v}\in V$, there holds
\begin{equation}
\label{eq:est6}
\left(\mathcal{L}w,\tilde{v}\right)= (v-P_N v,\tilde{v}).
\end{equation}

If we choose $\tilde{v}= v - P_N v$, since $I_N w$ and $v-P_N v$ are orthogonal, thanks to~\eqref{eq:est6}, we obtain
\begin{align}
\label{est4}
\norm{v-P_N v}_{L^2(\Omega)}^2 &= 
\left(\mathcal{L}w,v-P_N v\right) = \left(\mathcal{L}(w-I_N w),v-P_N v\right)\\
&\le \norm{w -I_N w}_V\norm{v-P_N v}_V. \notag
\end{align}
The interpolant estimate~\eqref{eq:est2} and the regularity estimate~\eqref{eq:reg} give us
\begin{equation}
\label{eq:est5}
\norm{w -I_N w}_V \le \frac{M}{N^2} \norm{w}_{L^2(\Omega)}\le\frac{M}{N^2} \norm{v-P_Nv}_{L^2(\Omega)}.
\end{equation}
Finally, using~\eqref{eq:est3}, we get the claim.
\end{proof}

We are now ready to prove Theorem~\ref{th:conv}.

\begin{proof}[Theorem~\ref{th:conv}]
Let $w_k$ be the eigenfunction of the continuous problem~\eqref{eq:variat} corresponding to the $k$-th smallest eigenvalue $\lambda_k$. We consider the $k$-dimensional subspaces
\[E_k=Span\{w_1,\cdots,w_k\},\quad \tilde{E}_k = P_N E_k.\]
If $v\in E_k$, with $v\ne0$, then
\begin{equation}
\label{eq:normPN}
\norm{P_N v}\ge \norm{v}-\norm{v-P_N v}\ge\left(1-\frac{M}{N^2}\right)\norm{v}>0.
\end{equation}
Therefore, $P_N$ is an isomorphism, and as a consequence the spaces $E_k$ and $\tilde{E}_k$ have the same dimension.

Thanks to Theorem~\ref{th:minmax}, we find
\begin{align*}
\lambda_k\le\lambda_N^k&\le \max_{\stackrel{v_N\in\tilde{E}_k}{v_N\ne0}}
\frac{\left(\mathcal{L}v_N,v_N\right)}{\norm{v_N}^2} = \max_{\stackrel{v\in E_k}{v\ne0}}\frac{\left(\mathcal{L}(P_N v),P_N v\right)}{\norm{P_N v}^2}\\
&\le\max_{\stackrel{v\in E_k}{v\ne0}}\frac{\left(\mathcal{L}v,v\right)}{\norm{P_N v}^2}\le\underbrace{\max_{\stackrel{v\in E_k}{v\ne0}}\frac{\left(\mathcal{L}v,v\right)}{\norm{v}^2}}_{{}=\lambda_k}\max_{\stackrel{v\in E_k}{v\ne0}}\frac{\norm{v}^2}{\norm{P_N v}^2}.
\end{align*}
Finally, using~\eqref{eq:normPN}, we get
\[\lambda_k\le\lambda_N^k\le\lambda_k\left(1+\frac{M}{N^2}\right),\]
and this concludes the proof.
\end{proof}

We can also derive a convergence result for the eigenfunctions.
\begin{theorem}
\label{th:conveigenf}
Let $(\lambda_k,w_k)$ and $(\lambda_N^k,w_N^k)$ be the $k$-th smallest eigenpair corresponding to the continuous problem~\eqref{eq:variat} and to the discrete problem~\eqref{eq:variat-discr}, respectively. Then, there exists a positive constant $M = M(\Omega)$ independent of $N$ such that
\begin{equation}
\label{eq:dist-eigenf}
\norm{w_k-w_N^k}_{L^2(\Omega)} \le \frac{M}{N^2}.
\end{equation}
\end{theorem}

\begin{proof}
For simplicity we assume that $\lambda_k$ and $\lambda_N^k$ have algebraic multiplicity equals to one. The proof of the general case is similar.

We define
\begin{equation}
\label{eq:vNk}
v_N^k = \left(P_N w_k, w_N^k\right)w_N^k,
\end{equation}
and
\begin{equation}
\label{eq:teta}
\theta_N^k = \max_{i\ne k} \frac{\lambda_k}{|\lambda_k-\lambda_N^i|}.
\end{equation}

We have
\begin{align}
\label{eq:est7}
\norm{w_k-w_N^k}_{L^2(\Omega)} &\le \norm{w_k - P_N w_k}_{L^2(\Omega)} +\norm{P_N w_k - v_N^k}_{L^2(\Omega)}\\
& + \norm{v_N^k -w_N^k}_{L^2(\Omega)}.\notag
\end{align}

Let us estimate the term $\norm{v_N^k -w_N^k}_{L^2(\Omega)}$. As long as we normalize the discrete and the continuous eigenfunctions, by~\eqref{eq:vNk} we have
\[\norm{v_N^k}_{L^2(\Omega)} = |\left(P_N w_k,w_N^k\right)|,\]
and
\begin{equation*}
v_N^k -w_N^k = \left(1-\left(P_N w_k,w_N^k\right)\right)w_N^k,
\end{equation*}
so that
\[
\norm{v_N^k -w_N^k}_{L^2(\Omega)} = |1-\left(P_N w_k,w_N^k\right)|.
\]
If we choose the sign of $w_N^k$ such that the inner product $\left(P_N w_k,w_N^k\right)\ge0$, then we obtain
\begin{align*}
\norm{v_N^k -w_N^k}_{L^2(\Omega)} &= |1-\left(P_N w_k,w_N^k\right)|\\
&=\bigl|1-|\left(P_N w_k,w_N^k\right)|\bigr|\\
&= \bigl|\norm{w_k}_{L^2(\Omega)} - \norm{v_N^k}_{L^2(\Omega)} \bigr|\\
&\le \norm{w_k- v_N^k}_{L^2(\Omega)}.
\end{align*}
Therefore,
\begin{align}
\label{eq:est8}
\norm{v_N^k -w_N^k}_{L^2(\Omega)} &\le \norm{w_k- v_N^k}_{L^2(\Omega)}\\
& \le \norm{w_k- P_N w_k}_{L^2(\Omega)} +\norm{ P_N w_k - v_N^k}_{L^2(\Omega)}\notag.
\end{align}

Thanks to~\eqref{eq:est8}, equation~\eqref{eq:est7} becomes
\begin{equation}
\label{eq:est9}
\norm{w_k-w_N^k}_{L^2(\Omega)} \le 2\norm{w_k - P_N w_k}_{L^2(\Omega)} +2\norm{P_N w_k - v_N^k}_{L^2(\Omega)}.
\end{equation}
Lemma~\ref{lm:L2proj} implies that there exists a positive constant $M=M(\Omega)$ independent of $N$ such that
\begin{equation}
\label{eq:est10}
\norm{w_k - P_N w_k}_{L^2(\Omega)}\le \frac{M}{N^2}.
\end{equation}
Finally, let $i\ne k$. Since
\begin{align*}
\left(P_N w_k, w_N^i\right) &=\frac{1}{\lambda_N^i}
\left(\mathcal{L}(P_N w_k),w_N^i\right) = \frac{1}{\lambda_N^i} \left(\mathcal{L}w_k,w_N^i\right)\\
&=\frac{\lambda_k}{\lambda_N^i}\left(w_k,w_N^i\right),
\end{align*}
we find
\begin{align*}
\left(w_k- P_N v_k,v_N^i\right)&= \left(w_k,w_N^i\right) - \left(P_N w_k,w_N^i\right)\\
&= \left(\frac{\lambda_N^i - \lambda_k}{\lambda_k}\right)\left(P_N w_k,w_N^i\right).
\end{align*}
So, using~\eqref{eq:teta}
\[
\bigl|\left(P_N w_k,w_N^i\right)\bigr| \le \theta_N^k \bigl|\left(w_k- P_N v_k,v_N^i\right)\bigr|.
\]
As a consequence,
\begin{align}
\label{eq:est11}
\norm{P_n w_k - v_N^k}_{L^2(\Omega)}^2 &=\sum_{i\ne k}\left(P_N w_k, w_N^i\right)^2\\
&\le (\theta_N^k)^2\sum_{i\ne k}\left(w_k-P_N w_k,w_N^i\right)^2\notag\\
&\le (\theta_N^k)^2\norm{w_k-P_N w_k}_{L^2(\Omega)}^2\le\frac{M}{N^4}.\notag
\end{align}
Collecting~\eqref{eq:est10} and~\eqref{eq:est11} into~\eqref{eq:est9}, we get the claim.
\end{proof}

In the following sections, we study the problem by considering the method of Fourier trigonometric and the Chebyshev polynomials as basis. The first choice is suitable in presence of periodic boundary condition, while the second one can be applied to non-periodic problem.

\begin{remark}
This work focuses on the one-dimensional peridynamic model, however, the variational approach can be easily extended to higher dimensional problems, by using other functional spaces and test functions. For instance, for the bi-dimensional case, we can implement the 2-D fast Fourier Transform as done in~\cite{LP2021}.
\end{remark}

\subsection{Spectral approximation by Fourier trigonometric polynomials}
\label{sec:Fourier}

Let $\Omega=[0,2\pi]$ be the computational domain. Assume that the problem is subject to periodic boundary conditions $w(0) = w(2\pi)$. 

Let $N>0$ be an even number and $\{\psi_k(x)\}_{|k|\le N/2}$ be an orthogonal basis of Fourier trigonometric polynomials in $L^2([0,2\pi])$ such that
\begin{equation}
\label{eq:trig}
\psi_k(x)=e^{ikx},\quad x\in\Omega,
\end{equation}
with $\psi_k(0)=\psi_k(2\pi)$, for all $k=-N/2,\cdots,N/2$.

We seek an approximation of $w(x)$ in form of real-valued Fourier trigonometric polynomials
\begin{equation}
\label{eq:Ftransftrig}
w_N(x)=\sum_{|h|\le N/2}\tilde{w}_h\psi_h(x).
\end{equation}

Then, using the same computation as before, namely, substituting~\eqref{eq:Ftransftrig} and~\eqref{eq:trig} into~\eqref{eq:variational}, we obtain the eigenvalue problem defined in~\eqref{eq:geneigpb}, where matrix values of $A$ and $M$ have the following expression
%
%
\begin{align*}
a_{hk}&=-\int_0^{2\pi}\psi_k(x)\int_{\Omega\cap B_{\delta}(x)}C(x-x')(\psi_h(x')-\psi_h(x))\,dx'dx\\
&=-\int_0^{2\pi}\psi_k(x)\psi_h(x)dx\,\int_{\Omega\cap B_{\delta}(0)}C(x')\cos(hx')dx'+2\pi\beta\delta_{h,-k}\\
&=-2\pi\delta_{h,-k}\int_{\Omega\cap B_{\delta}(0)}C(x')\cos(hx')dx'+2\pi\beta\delta_{h,-k},
\end{align*}
\[m_{hk}=\int_0^{2\pi}\psi_h(x)\psi_k(x)\,dx = 2\pi \delta_{h,-k},\]
and $\delta_{h,-k}$ is the Kronecker delta.

More in details, the matrices $A$ and $M$ have the following representation
\begin{equation*}
A=
\begin{bmatrix}
0& & *\\
& \iddots&\\
*& &0
\end{bmatrix},\qquad 
M=\begin{bmatrix}
0& & *\\
& \iddots&\\
*& &0
\end{bmatrix}.
\end{equation*}

We can observe that, in this case the matrices $A$ and $M$ are anti-diagonal and the product $M^{-1}A$ is a diagonal definite-positive matrix. Moreover, the integral $\gamma_h=\int_{\Omega\cap B_{\delta}(0)}C(x')\cos(hx')dx'$ which appear in the definition of $a_{hk}$ can be accurately computed by means of the Fast Fourier transform (FFT) as $C$ is an even real-valued function.

Additionally, when $h=0$, such integral is equal to $\beta$, therefore the matrix $A$ has a zero value on its diagonal, so $\lambda=0$ is an eigenvalue for the peridynamic operator. 

Since the matrices $A$ and $M$ are anti-diagonal in this setting, and since the product of two anti-diagonal matrices gives a diagonal matrix, a complete derivation of the spectrum of the peridynamic operator is possible in this case. Indeed we have
\[
\lambda_h = \beta - \gamma_h^N,\quad h=-N/2,\cdots,N/2,
\]
where $\gamma_h^N$ is the discrete Fourier transform of $\gamma_h$ and is less than $\beta$. They corresponds to the diagonal values of the matrix $M^{-1} A$. Moreover, the corresponding eigenfunctions are given by
\[
w_N^h(x) = \psi_h(x) = e^{ihx},\qquad x\in\Omega,\quad h=-N/2, \dots, N/2.
\]

Because the discrete sine and cosine functions of different frequencies are mutually orthogonal, we find $N$ linearly independent eigenfunctions. Note that, except for $\lambda_0$, the eigenvalues have multiplicity higher than one because of the symmetry between $h$ and $-h$.

We now compare the discrete eigenvalues with those of the continuous problem, found in~\eqref{eq:Fouriermultip}. The following theorem establishes the estimate for the distance between the continuous eigenvalue and its Fourier trigonometric approximation. This is a trivial consequence of the estimates of the distance between a function and its Fourier approximation.

\begin{theorem}[see~\cite{Canuto2006}]
\label{th:dist-trig}
Let $\lambda_k$ and $\lambda_N^k$ be the $k$-th smallest eigenvalues of the continuous and discrete eigenvalue problem respectively. If the micromodulus function $C\in L^2([0,2\pi])\cap\mathcal{C}^1([0,2\pi])$, then there exists a positive constant $M$ independent of $N$ such that
\begin{equation}
\label{eq:dist-trig}
|\lambda_k - \lambda_N^k| \le \frac{M}{N^2} \norm{C'}_{L^2([0,2\pi])}.
\end{equation}
\end{theorem}

\subsection{Spectral approximation by Chebyshev polynomials}
\label{sec:Chebyshev}
We consider the shifted Chebyshev polynomials $T_k(x)$ in the computational domain $\Omega=[0,2\pi]$, for $k=0,\cdots,N$, defined as
\[T_k(x)=\cos\left(k\arccos\left(\frac{x-\pi}{\pi}\right)\right).\]
They form an orthogonal basis in the weighted space $L^2_{\omega}([0,2\pi])$, where the weighted function is $\omega(x)= \left(\sqrt{1-\left(\frac{x-\pi}{\pi}\right)^2}\right)^{-1}$.

Indeed
\[\int_0^{2\pi}T_h(x)T_k(x)\omega(x)dx = \frac{\pi^2}{2}c_h\delta_{h,k},\]
where 
\[
c_h=\begin{cases}
2,&h=0,\\
1,&\text{otherwise}.
\end{cases}
\]

In this case, the equation~\eqref{eq:variational} becomes
\begin{align}
\label{eq:variat-cheby}
\int_0^{2\pi}&\left(-\int_{B_\delta(x)}C(x-x')\left(w(x')-w(x)\right)dx' \right)T_k(x)\omega(x)dx\\
&\qquad =\int_0^{2\pi}\lambda w(x) T_k(x)\omega(x)dx,\notag
\end{align}
and we approximate $w(x)$ as linear combination of the Chebyshev polynomials:
\begin{equation}
\label{eq:FtransfCheb}
w_N(x) = \sum_{h=0}^N\hat{w}_h T_h(x).
\end{equation}
If we substitute equation~\eqref{eq:FtransfCheb} into~\eqref{eq:variat-cheby}, we find again the eigenvalue problem
\[A\hat{w}=\lambda M\hat{w},\]
with matrices $A = (a_{hk})_{h,k}$ and $M=(m_{hk})_{h,k}$ defined as follows
\begin{equation*}
a_{hk}=-\int_0^{2\pi}\omega(x)T_k(x)\int_{\Omega\cap B_{\delta}(x)}C(x-x')T_h(x')dx'dx+\frac{\pi^2}{2}\beta c_h\delta_{h,k},
\end{equation*}
\[m_{hk}=\int_0^{2\pi}T_h(x)T_k(x)\omega(x)\,dx = \frac{\pi^2}{2}c_h \delta_{h,k}.\]
If we denote by 
\begin{equation}
\label{gammahx}
\gamma_h(x) = \int_{\Omega\cap B_{\delta}(x)}C(x-x')T_h(x')dx',
\end{equation}
we can observe that the integral $\int_0^{2\pi}\omega(x)\gamma_h(x) T_k(x) dx$ can be efficiently computed by using the Chebyshev trasform.

Additionally, using the Chebyshev polynomials for the approximation of the peridynamic operator, we have that the matrix $A$ is nondiagonal in general.

\begin{remark}
The advantage of using the Chebyshev polynomials relies in the fact that they can be implemented to solve problem without imposing periodic boundary conditions. However, in order to make a comparison between this method and the one based on the Fourier trigonometric polynomials, we can apply the Chebyshev polynomials to periodic problems. To do this, we need to require that $T_k(0) = T_k(2\pi)$, which implies that $k$ must be even.

Thus, we have to restrict our attention to the following orthogonal sub-basis $\{T_k(x)\}_k$, for $k=0, \dots, N$, $k$ even. 
\end{remark}


\section{Numerical simulations}
\label{sec:numerics}

In this section we validate and compare the methods proposed in Section~\ref{sec:Spectral} and collect some simulations in order to investigate the properties of the solutions of the eigenvalue problem~\eqref{eq:eigpb}.

To validate the results, we consider the following one-dimensional benchmark problem. We consider a bar in the spatial domain $[0,2\pi]$ and we take the micromodulus function $C(x) =4 e^{-x^2}/\sqrt{\pi}$, so in this case $\beta\approx 4$.

When we deal with the Fourier trigonometric polynomials, we discretize the computational domain by using the uniform mesh $x_j=2\pi j/N$ for $j=0,\cdots,N$; while, when we consider the Chebyshev polynomials, we take the non-uniform mesh given by the Guass-Chebyshev-Lobatto nodes $x_j=\pi + \pi\left(-\cos(\pi j/N)\right)$, for $j=0,\cdots,N$. This choice of nodes is suitable as it avoids the Gibb's phenomenon at the boundaries.

Figure~\ref{fig:eigv} shows a comparison between the exact eigenvalues of the peridynamic operator with the discrete ones computed by means of the Fourier trigonometric and Chebyshev polynomials for various choice of the horizon $\delta$. The convergence of the approximated eigenvalues is evaluated by computing the relative error in the discrete $L^2(\Omega)$ norm:
\[
E_{L^2} = \frac{\sum_{k}\bigl|\lambda_N^k- \lambda_k\bigr|^2}{\sum_{k}\bigl|\lambda_N^k\bigr|^2},
\]
where $\lambda_k$ represents the $k$-th smallest eigenvalue of the continuous problem~\eqref{eq:variat}, and $\lambda_N^k$ denotes $k$-th smallest eigenvalue of the discrete problem~\eqref{eq:variat-discr} computed using both the Fourier trigonometric and the Chebyshev polynomials.

Table~\ref{tab:conv} shows the relative error and the convergence rate for the two methods for different values of $N$ and for $\delta=3$. We can observe that the order of convergence is in accordance with the theoretical results. 

\begin{figure}
\centering
\begin{subfigure}[b]{.495\textwidth}
\includegraphics[width=\textwidth]{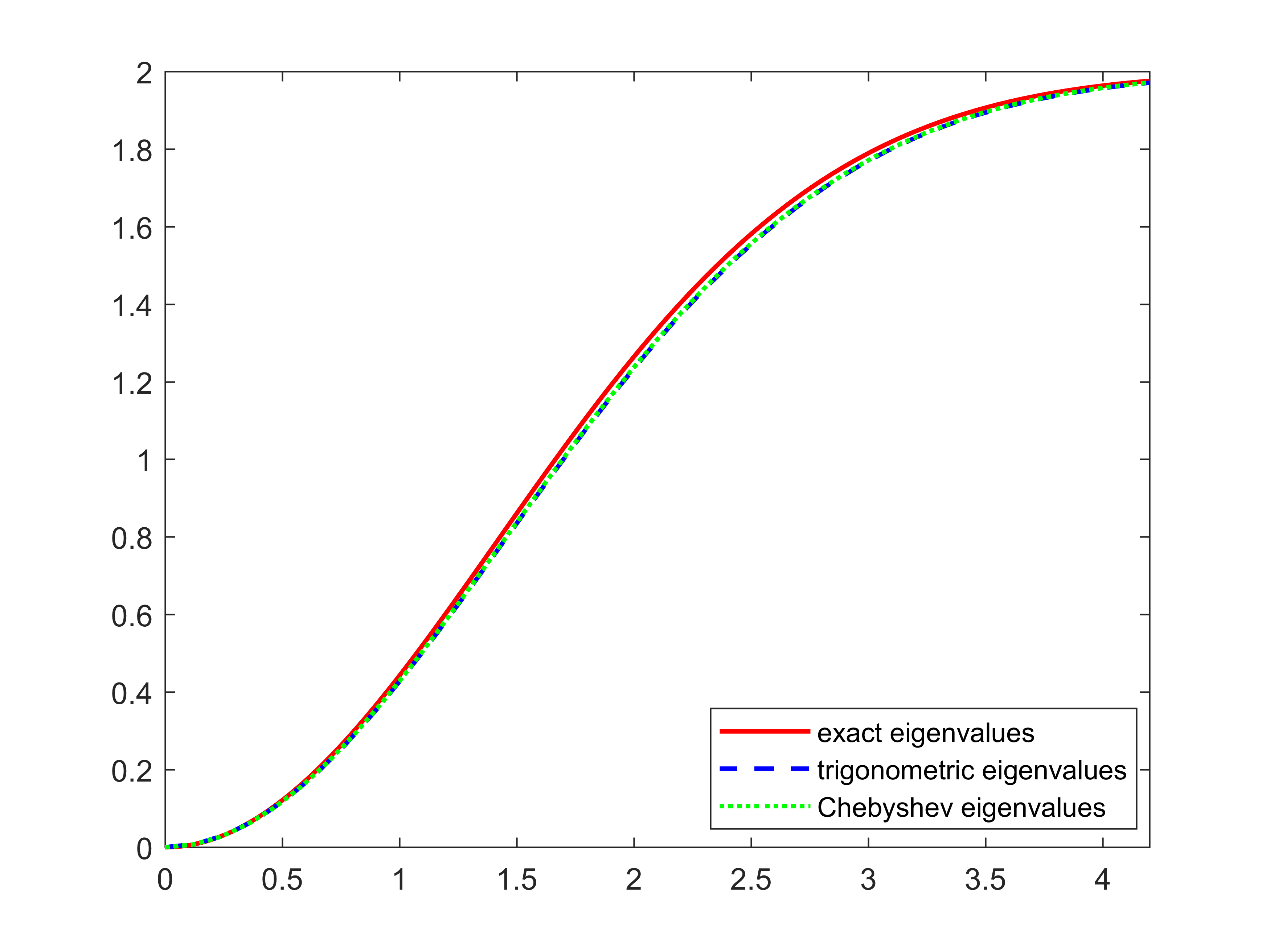}
\caption*{$\delta=3$.}
\end{subfigure}
\begin{subfigure}[b]{.495\textwidth}
\includegraphics[width=\textwidth]{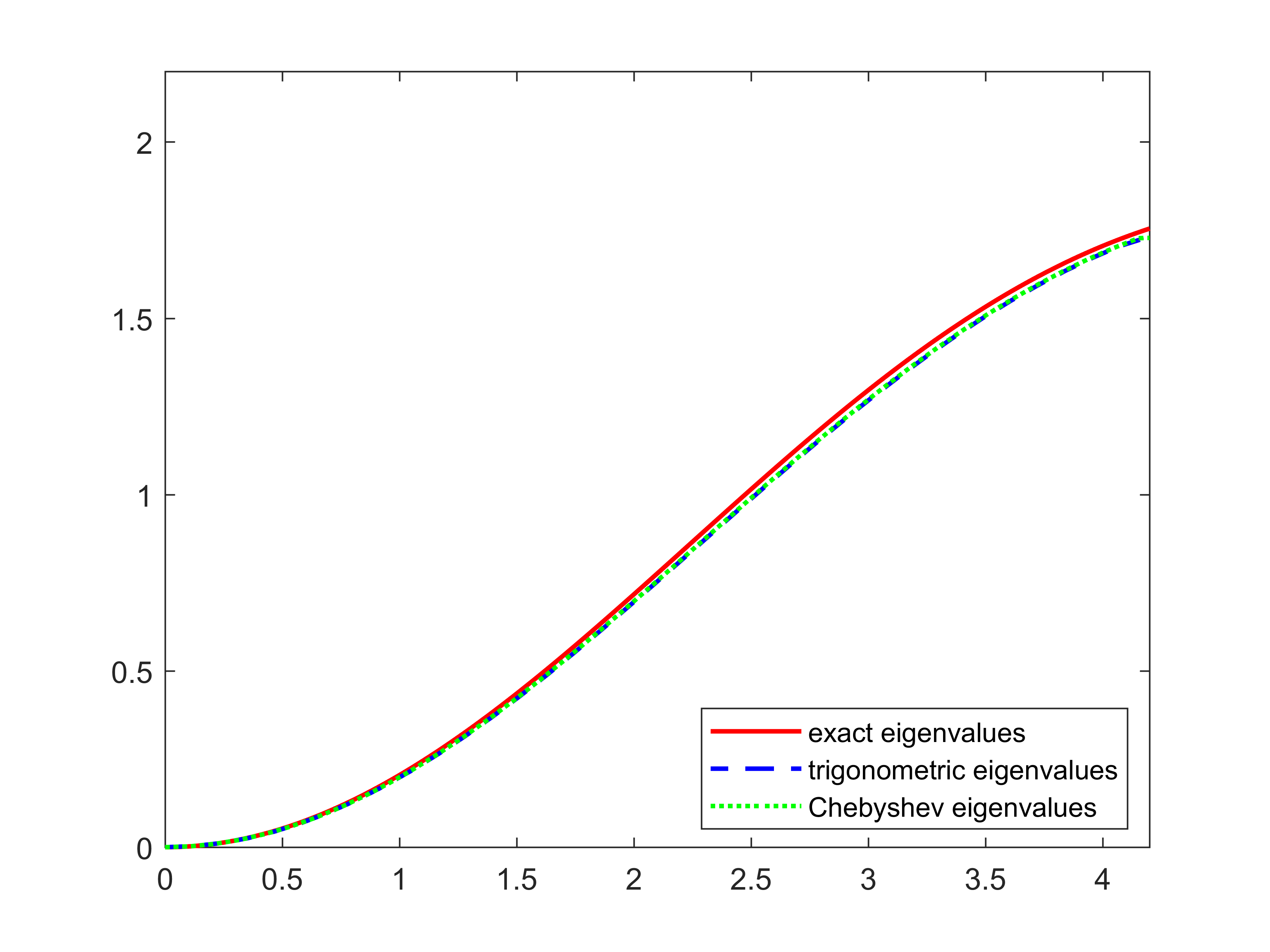}
\caption*{$\delta=1$.}
\end{subfigure}
\\
\begin{subfigure}[b]{.495\textwidth}
\includegraphics[width=\textwidth]{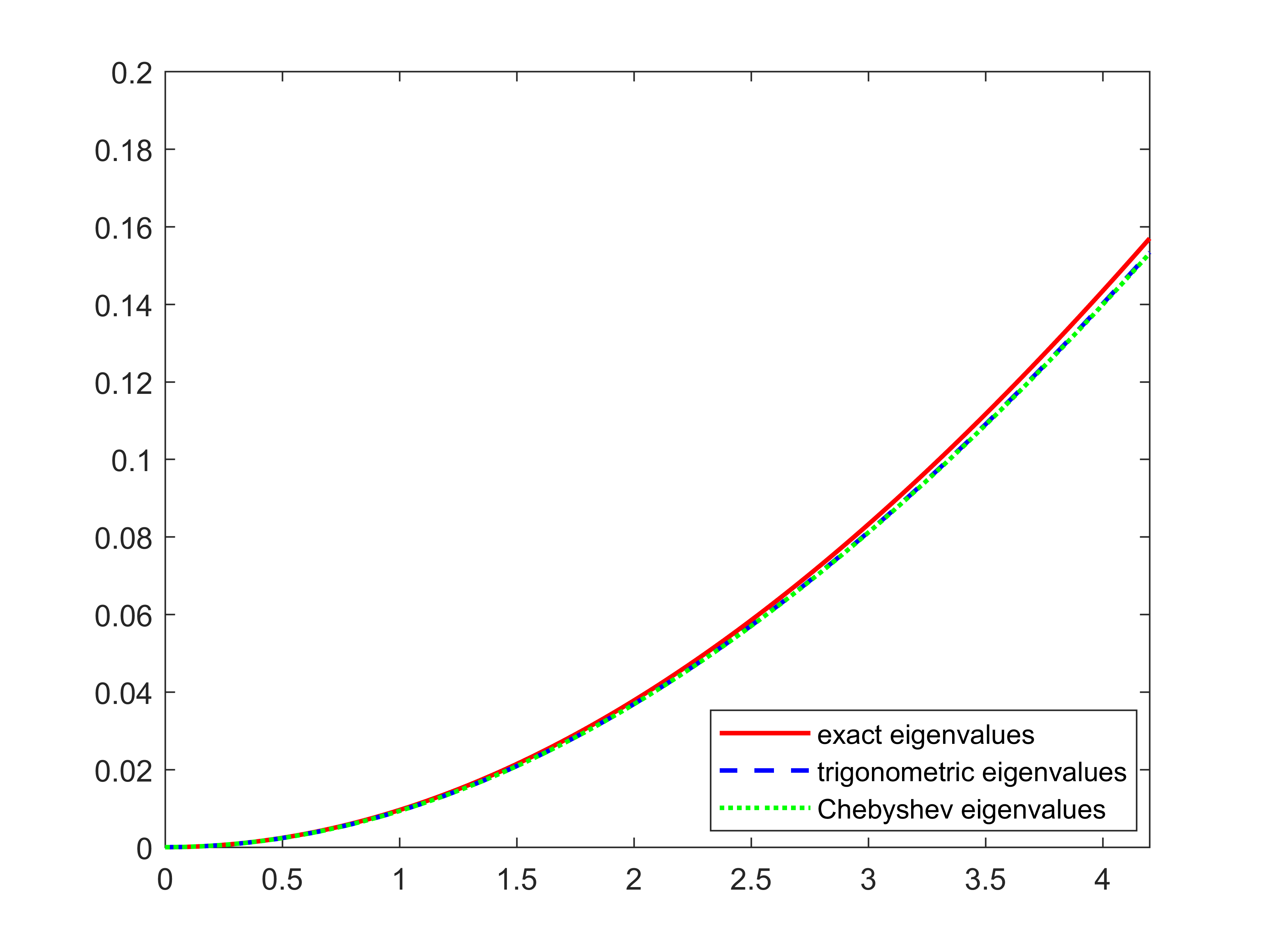}
\caption*{$\delta=0.3$.}
\end{subfigure}
\begin{subfigure}[b]{.495\textwidth}
\includegraphics[width=\textwidth]{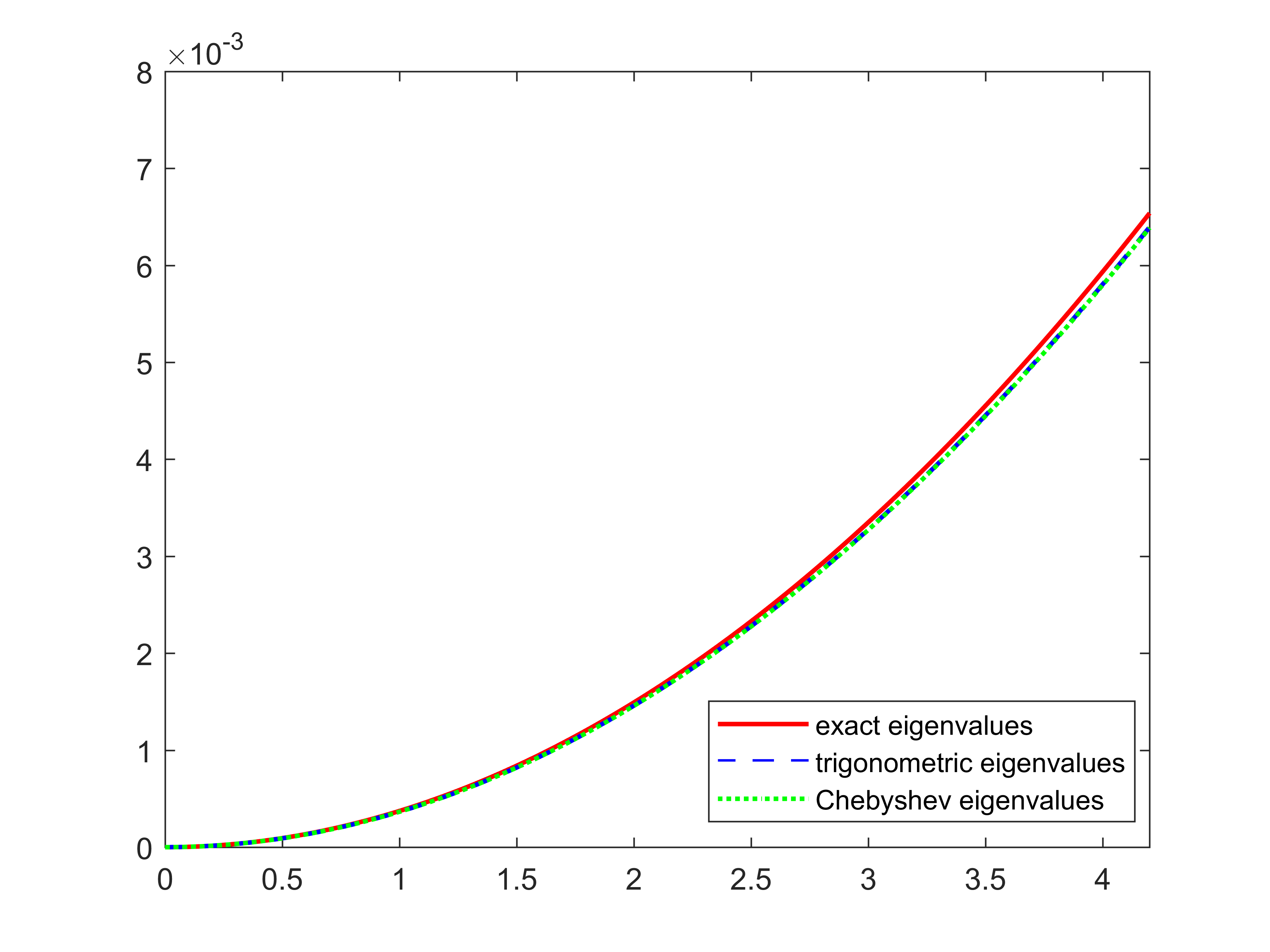}
\caption*{$\delta=0.1$.}
\end{subfigure}
\caption{Comparison of the exact peridynamic eigenvalues $\lambda(\nu_h)$ defined in~\eqref{eq:Fouriermultip} with respect to $\nu_h$ sampled at 100 equispaced points in the interval $[0,4]$ with the discrete ones computed by means of trigonometric and Chebyshev polynomials for different values of $\delta$. For the simulation we fix $N=100$.}
\label{fig:eigv}      
\end{figure}

\begin{table}
\centering
\begin{tabular}{|ccc|ccc|}
\toprule
\multicolumn{3}{|c|}{Trigonometric approximation} & \multicolumn{3}{c|}{Chebyshev approximation}\\
\hline
$N$ & $E_{L^2}$ & convergence rate & $N$ & $E_{L^2}$ & convergence rate\\
\hline
$20$& $4.8016\times 10^{-3}$& $-$& $20$ & $8.8896\times 10^{-3}$& $-$\\
$40$& $1.1021\times 10^{-3}$& $2.0465$& $40$ & $8.6192\times 10^{-4}$& $2.0377$\\
$80$& $2.5438\times 10^{-4}$& $2.0616$& $80$ & $2.9249\times 10^{-4}$& $2.0485$\\
$160$& $5.5865\times 10^{-5}$& $2.0943$& $160$ & $5.1822\times 10^{-5}$&$2.0867$\\
$320$& $1.1000\times 10^{-5}$& $2.1481$& $320$ & $1.1802\times 10^{-5}$&$2.1267$\\
\bottomrule
\end{tabular}
\caption{The relative $L^2$-error as function of $N$ corresponding to the trigonometric and the Chebyshev approximation.}
\label{tab:conv}
\end{table}

Table~\ref{tab:eig} depicts the eigenvalues computed by the two methods for $N=20$ and $\delta =3$. We can notice that they are all positive and bounded by the value of $\beta$.

Moreover, we find that using the Fourier trigonometric polynomials as basis for the approximation, except for $\lambda_0\approx 0$, the algebraic multiplicity of the other eigenvalues is equal to two. This is correlated with the fact that trigonometric polynomials are characterized by two different frequencies.

\begin{table}
\centering
\label{tab:eig-comparison} 
\begin{tabular}{|c|c|}
\hline
$\lambda_n^{Trig}$, $n=0,\cdots, 20$  & $\lambda_n^{Cheby}$, $n=0,\cdots, 20$  \\
\hline
0.0000 & 0.1905 \\
0.8846 & 0.7084 \\
0.8846 & 1.4173 \\
2.5285 & 2.1567 \\
2.5285 & 2.8007 \\
3.5782 & 3.2864 \\
3.5787 & 3.6105 \\
3.9267 & 3.8042 \\
3.9267 & 3.9090 \\
3.9921 & 3.9607 \\
3.9921 & 3.9841 \\
3.9995 & 3.9940 \\
3.9995 & 3.9978 \\
3.9999 & 3.9993 \\
3.9999 & 3.9997 \\
3.9999 & 3.9999 \\
3.9999 & 3.9999 \\
3.9999 & 3.9999 \\
3.9999 & 3.9999 \\
3.9999 & 3.9999 \\
3.9999 & 3.9999 \\
\hline
\end{tabular}
\caption{The peridynamic eigenvalues computed by using trigonometric and Chebyshev polynomials}
\label{tab:eig}
\end{table}

Figure~\ref{fig:trig-eigfun} shows the first five eigenfunctions obtained by using the trigonometric polynomials for the approximation of the peridynamic operator.

\begin{figure}
\centering
\includegraphics[width=.65\textwidth]{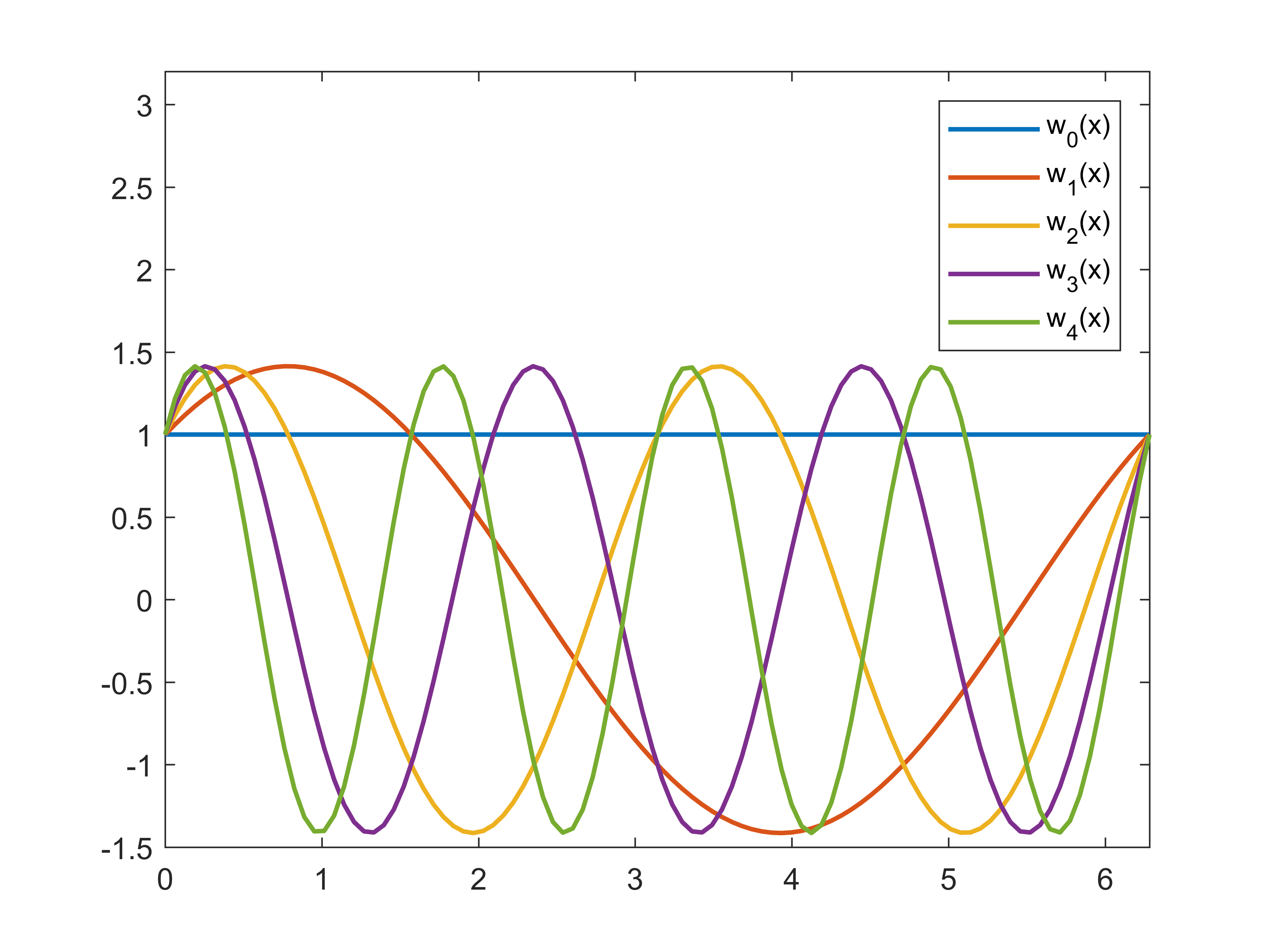}
\caption{The first five eigenfunctions obtained by using the trigonometric polynomials.}
\label{fig:trig-eigfun}      
\end{figure}

In addition, we make a simulation considering the peridynamic eigenvalue problem subject to antiperiodic boundary condition on the spatial domain $\Omega=[-1,1]$:
\begin{equation}
\label{eq:noperiod}
\begin{cases}
-\mathcal{L}w(x) = \lambda w(x),\qquad x\in [-1,1],\\
w(-1)=-1,\,w(1)=1.
\end{cases}
\end{equation}
The well-posedness of~\eqref{eq:noperiod} is achieved in~\cite{zhou,valdinoci}.
 For the simulation we use the same micromodulus function as before and we compute the eigenvalues by using the Chebyshev polynomials. In Figure~\ref{fig:nonperiodic}, we display a comparison between the exact eigenvalues and the discrete ones for different values of $N$ and $\delta$. Here, the exact eigenvalues are computed by using the proposed numerical method by using $N=200$.

\begin{figure}
\centering
\begin{subfigure}[b]{.495\textwidth}
\includegraphics[width=\textwidth]{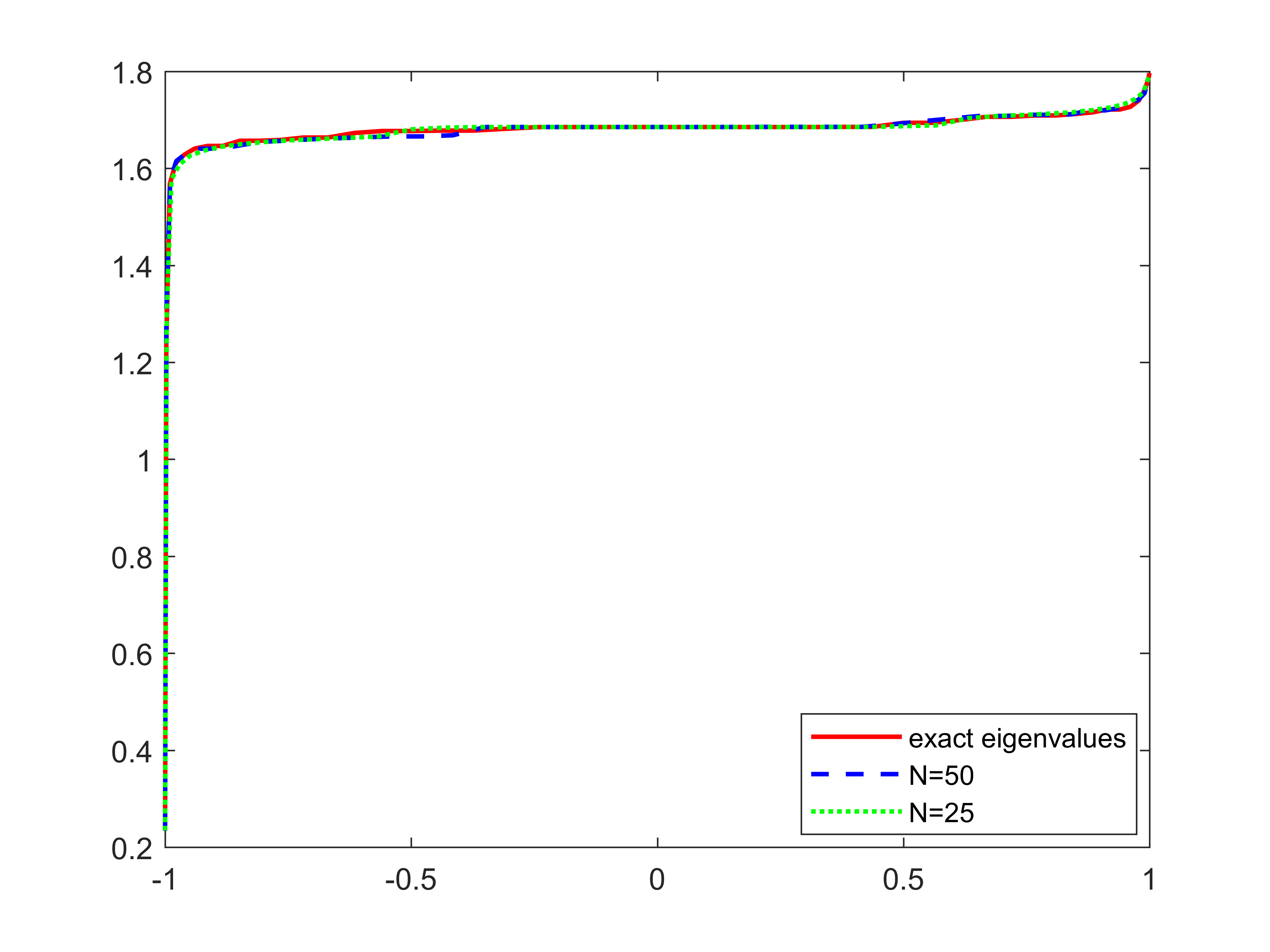}
\caption*{$\delta=1$.}
\end{subfigure}
\begin{subfigure}[b]{.495\textwidth}
\includegraphics[width=\textwidth]{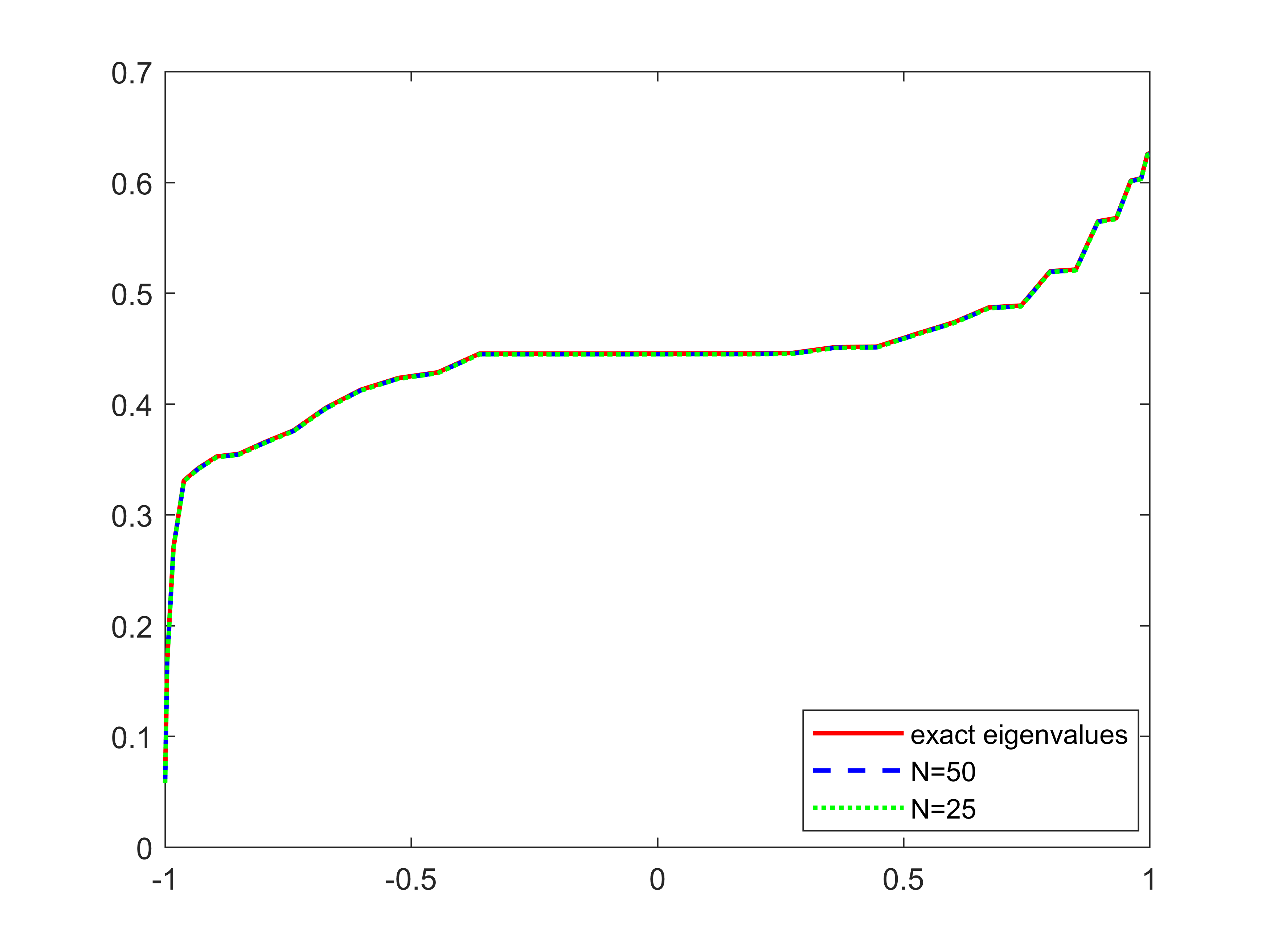}
\caption*{$\delta=0.2$.}
\end{subfigure}
\caption{The comparison of the exact eigenvalues of the problem~\eqref{eq:noperiod} in the interval $[-1,1]$ with the discrete ones computed by means of the Chebyshev polynomials for different choice of $N$ and $\delta$.}
\label{fig:nonperiodic}      
\end{figure}

\subsection{The solution of the peridynamic problem}

Let $[0,T]$, for some $T>0$, be the time domain under investigation. Consider a microelastic bar with constant mass density $\rho>0$ occupying a region $\Omega\subset {\mathbb{R}}$. Then, the peridynamic model describes the dynamics of the body and its equation is given by
\begin{equation}
\label{eq:perid}
\rho\partial_{tt}^2{ u}( x,t) = \mathcal{L} u(x,t) = \int_{\Omega\cap B_{\delta}(x)}  C( x'- x)\left( u( x',t) -  u( x,t)\right)d x',
\end{equation}
for $x\in \Omega$, and $t\in[0,T]$, with initial conditions
\begin{equation}
\label{eq:initcond}
 u( x,0)= u_0(x),\quad \partial_t  u( x,0)= v_0( x),\qquad  x\in \Omega,
\end{equation}
where $u$ is the displacement field.

We assume that the bar is subjected to the uniform initial displacement $u_0(x) = e^{-x^2}$, $v_0(x)=0$ and we fix $\delta =3$ as the size of the horizon. 

Moreover, we assume that the constant density of the body is $\rho(x) = 1$ and the interaction between two material particle is given by the micromodulus function $C$ defined as in the previous section.

If the spectrum of the nonlocal operator $-\mathcal{L}$ is known, one can formally express the solution of~\eqref{eq:perid}-\eqref{eq:initcond} as
\begin{equation}
\label{eq:formal-solution}
u(x,t) =\sum_{k=0}^{\infty} w_k(x)\left(a_k\cos(\sqrt{\lambda_k} t) + b_k\sin(\sqrt{\lambda_k} t)\right),\qquad x\in\Omega,\quad t\in[0,T],
\end{equation}
where $\lambda_k$ and $w_k(x)$ represent respectively the $k$-th eigenvalue and eigenfunction of the peridynamic operator $-\mathcal{L}$ on the domain $\Omega$, and the coefficients $a_k$ and $b_k$ are computed from the initial conditions~\eqref{eq:initcond} by
\begin{align}
\label{eq:a-b-coeff}
a_k &= \frac{\int_{\Omega} u_0(x)w_k(x)dx}{\int_{\Omega} w_k^2(x)dx},\\
b_k &= \frac{1}{\sqrt{\lambda_k}}\frac{\int_{\Omega} v_0(x) w_k(x)dx}{\int_{\Omega} w_k^2(x)dx}.\notag
\end{align}

Thus, the problem of finding the solution of~\eqref{eq:perid}-\eqref{eq:initcond} reduces to study an eigenvalue problem.

\begin{figure}
\centering
\begin{subfigure}[b]{.495\textwidth}
\includegraphics[width=\textwidth]{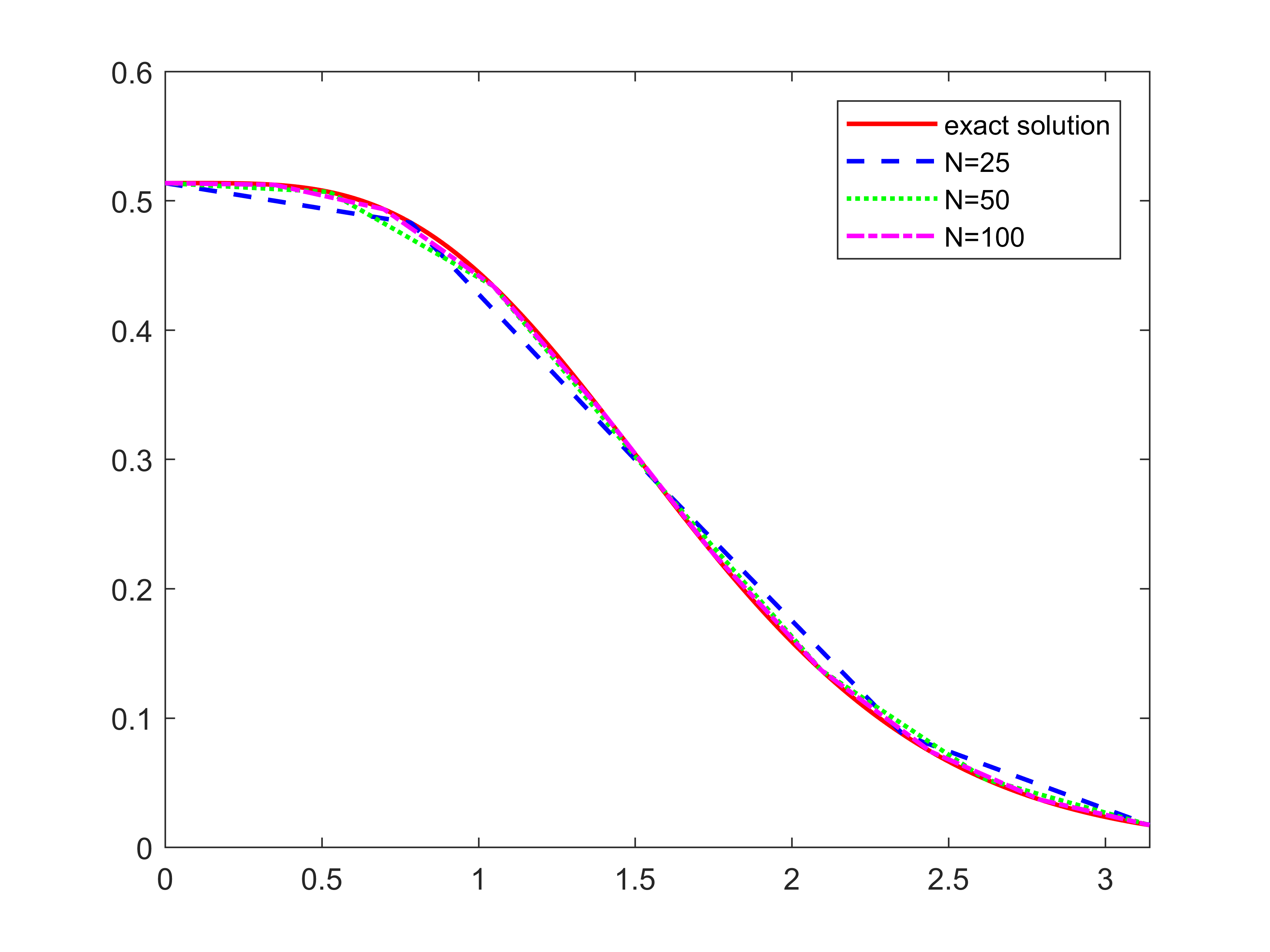}
\end{subfigure}
\begin{subfigure}[b]{.495\textwidth}
\includegraphics[width=\textwidth]{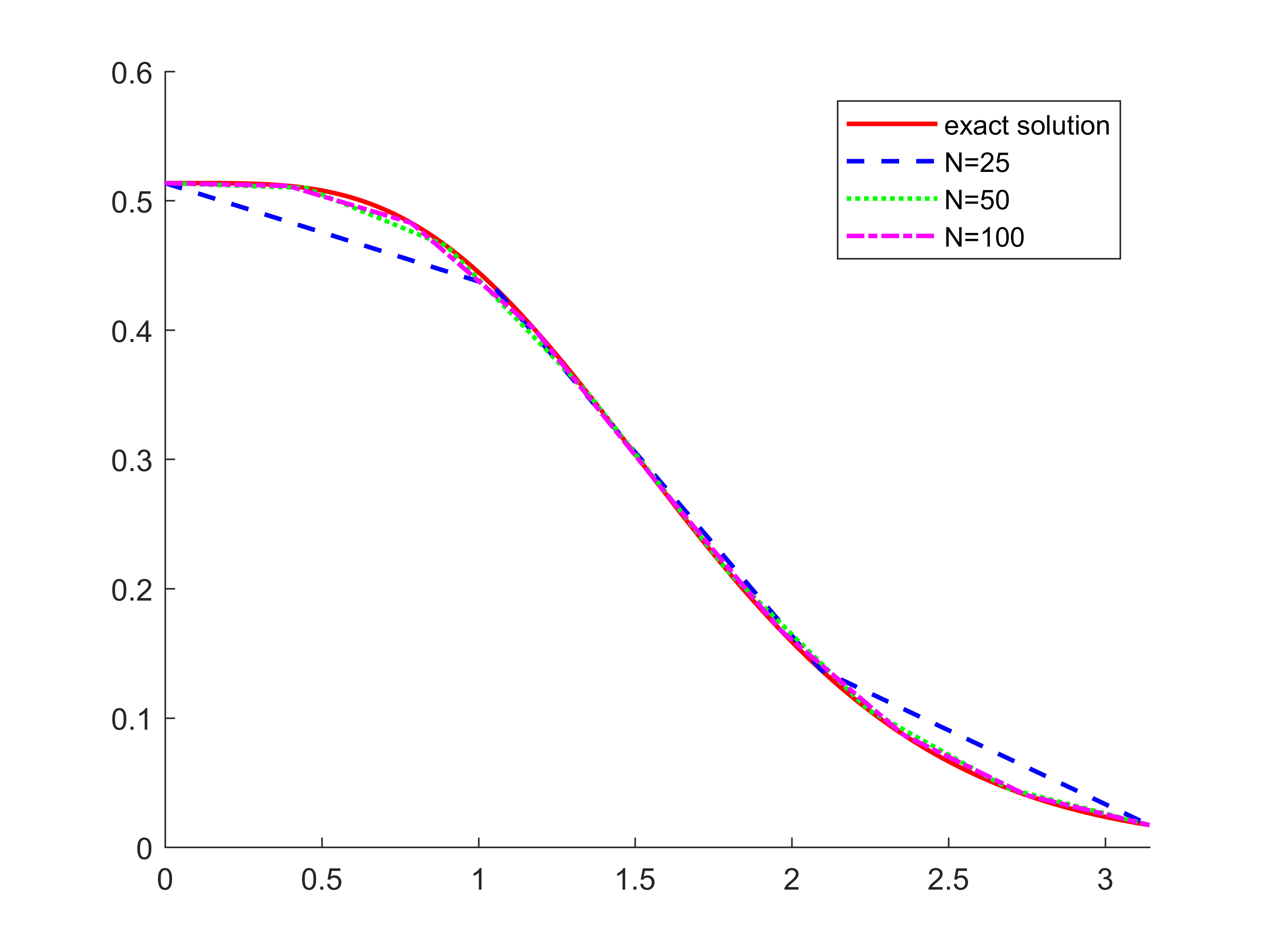}
\end{subfigure}
\caption{The comparison between the exact solution and the numeric one for different values of $N$ at time $t=1$. On the left, the numeric solution is computed by using the trigonometric polynomials; on the right, the numeric solution is computed by using the Chebyshev polynomials.}
\label{fig:compare}      
\end{figure}

In Figure~\ref{fig:compare} we make a comparison between the exact solution of the peridynamic problem~\eqref{eq:perid}-\eqref{eq:initcond} with the numerical solution obtained by using both the Fourier trigonometric polynomials and the Chebyshev polynomials as basis to solve the eigenvalue problem~\eqref{eq:geneigpb}. We can observe in both cases a good agreement as the total number of mesh-points $N$ increases.

\section{Conclusions}
\label{sec:concl}

In this paper we propose two spectral methods to numerically compute the eigenvalues of the one-dimensional nonlocal linear peridynamic operator. They deal with the weak formulation of peridynamic model and make use of both the Fourier trigonometric and the Chebyshev polynomials. The first ones are suitable in presence of periodic boundary conditions, while the others can be applied even to non periodic problems. We prove the convergence of the discrete eigenpairs to the continuous ones in the $L^2$-norm and our simulations confirm the theoretical results.

\section*{Acknowledgements}
This paper has been partially supported by GNCS of Istituto Nazionale di Alta Matematica, by PRIN 2017 ``Discontinuous dynamical systems: theory, numerics and applications'' coordinated by Nicola Guglielmi and by Regione Puglia, ``Programma POR Puglia 2014/2020-Asse X-Azione 10.4 Research for Innovation-REFIN - (D1AB726C)''. The authors thank the anonymous referees for critical remarks, suggestions and for the careful reading of the manuscript.


\bibliographystyle{plain}
\bibliography{biblioPeri2}

\end{document}